\documentclass[a4paper,12pt,reqno]{amsart}

\usepackage{amsmath}
\usepackage{amscd}
\usepackage{amssymb}
\usepackage{mathrsfs}
\usepackage[left=2.5cm,right=2.5cm,bottom=3cm,top=3cm]{geometry}
\newtheorem{thm}{Theorem}[section]
\newtheorem{cor}[thm]{Corollary}

\newtheorem{lem}[thm]{Lemma}

\newtheorem{prop}[thm]{Proposition}
\theoremstyle{definition}

\theoremstyle{remark}
\newtheorem{rem}[thm]{Remark}

\numberwithin{equation}{section}
\begin{document}

\title{A note on conical K\"ahler-Ricci flow on minimal elliptic K\"ahler surfaces}

\author{Yashan Zhang}
\address{Department of mathematics, University of Macau, Macau, China}
\email{yashanzh@163.com}
\thanks{The author is partially supported by the Science and Technology Development Fund (Macao S.A.R.) Grant FDCT/ 016/2013/A1 and the Project MYRG2015-00235-FST of the University of Macau}

\begin{abstract}
We prove that, under a semi-ampleness type assumption on the twisted canonical line bundle, the conical K\"{a}hler-Ricci flow on a minimal elliptic K\"{a}hler surface converges in the sense of currents to a generalized conical K\"{a}hler-Einstein on its canonical model. Moreover, the convergence takes place smoothly outside the singular fibers and the chosen divisor.
\end{abstract}

\maketitle

\section{Introduction}
The K\"{a}hler-Ricci flow has been studied extensively and become a powerful tool in K\"{a}hler geometry. Cao \cite{C} proved that K\"{a}hler-Ricci flow will converge smoothly to K\"{a}hler-Einstein metrics on a compact K\"{a}hler manifold with numerically trivial or ample canonical line bundle. Tsuji \cite{Ts} and Tian-Zhang \cite{TZo} proved the convergence of K\"{a}hler-Ricci flow to singular K\"{a}hler-Einstein metrics on a smooth minimal model of general type. In \cite{ST06,ST12}, Song-Tian obtained the weak convergence of K\"{a}hler-Ricci flow to generalized K\"{a}hler-Einstein metrics if the canonical line bundle is semi-ample.

The conical K\"{a}hler-Ricci flow was introduced in \cite{CW15}  and is expected to deform a conical K\"{a}hler metric to conical K\"{a}hler-Einstein metric. In fact the convergence of the conical K\"{a}hler-Ricci flow has been studied when the twisted canonical line bundle is positive or trivial \cite{CW}, nef and big \cite{Sh} or negative \cite{LZ}. Also see, e.g., \cite{MRS,PSSW1,Yi} for discussions of the conical Ricci flow on Riemann surfaces.
\par This paper aims to study the convergence of the conical K\"{a}hler-Ricci flow on a minimal elliptic K\"{a}hler surface of Kodaira dimension one, under a semi-ampleness type assumption (\ref{assumption}) on the twisted canonical line bundle.
\par Throughout this paper, let $(X,\omega_{0})$ be a minimal elliptic K\"{a}hler surface of Kodaira dimension $kod(X)=1$. By definition (see, e.g., Section I.3 of \cite{M} or Section 2.2 of \cite{ST06}), there exists a holomorphic map $f:X\rightarrow\Sigma$, determined by the pluricanonical system $|mK_{X}|$ for sufficiently large integer $m$, from $X$ onto a smooth projective curve $\Sigma$ (i.e., the canonical model of $X$), such that the general fiber is a smooth elliptic curve and all fibers are free of $(-1)$-curves. Set $\Sigma_{reg}:=\{s\in \Sigma|X_{s}:=f^{-1}(s)$ is a nonsingular fiber$\}$ and $X_{reg}=f^{-1}(\Sigma_{reg})$. Assume  $\Sigma\setminus\Sigma_{reg}=\{s_{1},\ldots,s_{k}\}$ and let $m_{i}F_{i}=X_{s_{i}}$ be the corresponding singular fiber of multiplicity $m_{i}$, $i=1,\ldots,k$. We refer readers to Section I.5 of \cite{M} for several interesting examples of minimal elliptic surfaces.
\par Consider a fixed point $r\in\Sigma_{reg}$, which can be seen as a divisor on $\Sigma$, and the smooth divisor $D:=X_{r}$ on $X$. Let $S'$ be the defining section of $r$ and $h'$ be a fixed smooth Hermitian metric on the holomorphic line bundle associated to $r$ with $|S'|_{h'}^{2}\leq1$. Then $S:=f^{*}S'$ is a defining section of $D$ and $h:=f^{*}h'$ is a smooth Hermitian metric on the holomorphic line bundle associated to $D$ with $|S|_{h}^{2}=f^{*}|S'|_{h'}^{2}\leq1$. Set $\Gamma'=\{r,s_{1},\ldots,s_{k}\}$ and $\Gamma=f^{-1}(\Gamma')$.
\par Fix an arbitrary $\beta\in(0,1)$. Then for any sufficiently small positive constant $\delta$,
\begin{equation}
\omega_{0}^{*}:=\omega_{0}+\delta\sqrt{-1}\partial\bar{\partial}|S|_{h}^{2\beta}\nonumber
\end{equation}
is a conical K\"{a}hler metric with cone angle $2\pi\beta$ along $D$ (see \cite{Do}), which means that $\omega_{0}^{*}$ is a smooth K\"{a}hler metric on $X\setminus D$ and is asymptotically equivalent along $D$ to the local model metric
\begin{equation}
\sqrt{-1}\left(\frac{dz_{1}\wedge d\bar{z}_{1}}{|z_{1}|^{2(1-\beta)}}+dz_{2}\wedge d\bar{z}_{2}\right)\nonumber,
\end{equation}
on $\mathbb{C}^{2}$, where $(z_{1},z_{2})$ are local holomorphic coordinates such that $D=\{z_{1}=0\}$ locally.
\par Consider the conical K\"{a}hler-Ricci flow
\begin{equation}\label{CKRF}
\left\{
\begin{aligned}
\partial_{t}\omega&=-Ric(\omega)-\omega+2\pi(1-\beta)[D]\\
\omega(0)&=\omega_{0}^{*},
\end{aligned}
\right.
\end{equation}
where $[D]$ is the current of integration associated to the divisor $D$ on $X$. By results in \cite{Sh} we know that a solution $\omega$ to (\ref{CKRF}) exists uniquely for all $t\geq0$ if and only if the twisted canonical line bundle $K_{X}+(1-\beta)L_{D}$ is nef, where $L_{D}$ is the holomorphic line bundle associated to the divisor $D$. In particular, under the assumption (\ref{assumption}) below, (\ref{CKRF}) admits a unique long time solution. The following is our main result.
\begin{thm}\label{main}
Assume as above. Assume in addition that there exists a K\"{a}hler metric $\theta'$ on $\Sigma$ such that
\begin{equation}\label{assumption}
f^{*}\theta'\in -2\pi c_{1}(X)+2\pi(1-\beta)[D].
\end{equation}
Then,
\begin{itemize}
\item[(1)] As $t\to\infty$, $\omega(t)\to f^{*}\omega_{\infty}'$ as currents on $X$ and the convergence takes place smoothly on $X\setminus\Gamma$.
\item[(2)] $\omega_{\infty}'$ is a positive closed $(1,1)$-current on $\Sigma$ such that it is a smooth K\"{a}hler metric on $\Sigma\setminus\Gamma'$ and $Ric(\omega_{\infty}')=-\sqrt{-1}\partial\bar{\partial}\log\omega_{\infty}'$ is a well-defined current on $\Sigma$ satisfying
\begin{equation}
Ric(\omega_{\infty}')=-\omega_{\infty}'+2\pi(1-\beta)[r]+\omega_{WP}+2\pi\sum_{i=1}^{k}\frac{m_{i}-1}{m_{i}}[s_{i}],
\end{equation}
where $\omega_{WP}$ is the induced Weil-Petersson metric.
\end{itemize}
\end{thm}

\begin{rem}
Theorem \ref{main} is a generalization of the results of Song-Tian \cite{ST06} and Fong-Zhang \cite{FZ} on the K\"{a}hler-Ricci flow to the conical setting. Following \cite{ST06}, we may call $\omega_{\infty}'$ a \emph{generalized conical K\"{a}hler-Einstein metric} on $\Sigma$.
\end{rem}
\begin{rem}
We should mention that in the setting of Theorem \ref{main}, it is proved in \cite{Ed} that the scalar curvature of $\omega(t)$ is uniformly bounded on $(X\setminus D)\times[0,\infty)$.
\end{rem}
\par The paper is organized as follows: In Section 2 we construct the generalized conical K\"{a}hler-Einstein metric $\omega_\infty'$ on $\Sigma$. In Section 3 we derive some necessary estimates and then give a proof of Theorem \ref{main}.

\section{Generalized conical K\"{a}hler-Einstein metrics}
In this section we construct the generalized conical K\"{a}hler-Einstein metric on $\Sigma$. To begin with, let $\omega_{SF}:=\omega_{0}+\sqrt{-1}\partial\bar{\partial}\rho_{SF}$, where $\rho_{SF}$ is a smooth function on $X_{reg}$, be the semi-flat $(1,1)$-form defined by Lemma 3.2 of \cite{ST06}. Here the semi-flatness means that, for any $s\in\Sigma_{reg}$, the restriction $\omega_{SF}|_{X_{s}}$ is a Ricci-flat K\"{a}hler metric on the smooth fiber $X_{s}$. On the other hand, by the assumption (\ref{assumption}) we can fix a smooth volume form $\Omega$ on $X$ satisfying
\begin{equation}
\int_{X}\Omega=\int_{X}\omega_{0}^{2}\nonumber.
\end{equation}
and
\begin{equation}
\sqrt{-1}\partial\bar{\partial}\log\Omega=f^{*}\theta'-(1-\beta)R_{h}\nonumber
\end{equation}
where $R_{h}=-\sqrt{-1}\partial\bar{\partial}\log h$ is the curvature form of $h$. Note that if $R_{h'}$ is the curvature form of $h'$, then $R_{h}=f^{*}R_{h'}$.
Define
\begin{equation}
F=\frac{\Omega}{2\omega_{SF}\wedge f^{*}\theta'}.
\end{equation}
Then $F$ can be seen as a function on $\Sigma$ and $F\in L^{p}(\Sigma,\theta)$ for some $p>1$ (see \cite{ST06,ST12,He}). Without loss of any generality we assume $p<\frac{1}{1-\beta}$.
\par The following proposition is essentially contained in \cite{ST06}, which is also a special case of the general theory of complex Monge-Amp\`{e}re equations (see \cite{Ko1,Ko2}).
\begin{prop}\label{prop1}
There exists a unique solution $\varphi_{\infty}$ solving the following equation on $\Sigma$
\begin{equation}\label{MA}
\theta'+\sqrt{-1}\partial\bar{\partial}(\varphi_{\infty}+\delta|S'|_{h'}^{2\beta})=\frac{Fe^{\varphi_{\infty}+\delta|S'|_{h'}^{2\beta}}}{|S'|_{h'}^{2(1-\beta)}}\theta'
\end{equation}
with $\varphi_{\infty}+\delta|S'|_{h'}^{2\beta}\in PSH(\Sigma,\theta)\bigcap C^{\alpha}(\Sigma,\theta)\bigcap C^{\infty}(\Sigma\setminus\Gamma')$ for some $\alpha\in(0,1)$.
\par Moreover, if we define $\omega_{\infty}'=\theta'+\sqrt{-1}\partial\bar{\partial}(\varphi_{\infty}+\delta|S'|_{h'}^{2\beta})$, then $\omega_{\infty}'$ satisfies all the properties stated in part (2) of Theorem \ref{main}.
\end{prop}

\begin{rem}
$\omega_{\infty}'$ is a conical K\"{a}hler metric on $\Sigma_{reg}$ with cone angle $2\pi\beta$ at $r$.
\end{rem}

\section{Estimates and convergence}
To prove part (1) of Theorem \ref{main} we first reduce the conical K\"{a}hler-Ricci flow (\ref{CKRF}) to a parabolic complex Monge-Amp\`{e}re equation, as in, e.g., \cite{LZ,Sh,Ed}.
\par Set $\omega_{t}=e^{-t}\omega_{0}+(1-e^{-t})f^{*}\theta'$. Then $\omega=\omega_{t}+\delta\sqrt{-1}\partial\bar{\partial}|S|_{h}^{2\beta}+\sqrt{-1}\partial\bar{\partial}\varphi$ solves the conical K\"{a}hler-Ricci flow (\ref{CKRF}) if $\varphi=\varphi(t)$ solves the following parabolic complex Monge-Amp\`{e}re equation
\begin{equation}\label{CKRF'}
\left\{
\begin{aligned}
\partial_{t}\varphi&=\log\frac{e^{t}|S|_{h}^{2(1-\beta)}(\omega_{t}+\delta\sqrt{-1}\partial\bar{\partial}|S|_{h}^{2\beta}+\sqrt{-1}\partial\bar{\partial}\varphi)^{2}}{\Omega}-\varphi-\delta|S|_{h}^{2\beta}\\
\varphi(0)&=0.
\end{aligned}
\right.
\end{equation}
Note that in the above reduction we have used the Poincar\'{e}-Lelong formula
\begin{equation}
\sqrt{-1}\partial\bar{\partial}\log|S|_{h}^{2}=-R_{h}+2\pi[D]\nonumber.
\end{equation}
To obtain the estimates we need, we will make use of the approximation method developed in \cite{CGP}, which is also used in, e.g., \cite{LZ,Sh,Ed}. Following \cite{CGP} we define the smoothing metric by
\begin{equation}
\omega_{t,\epsilon}=\omega_{t}+\delta\sqrt{-1}\partial\bar{\partial}\chi(\epsilon^{2}+|S|_{h}^{2})\nonumber,
\end{equation}
where
\begin{equation}
\chi(\epsilon^{2}+x)=\beta\int_{0}^{x}\frac{(\epsilon^{2}+r)^{\beta}-\epsilon^{2\beta}}{r}dr\nonumber.
\end{equation}
Note that for each $\epsilon>0$, $\omega_{t,\epsilon}$ is a smooth K\"{a}hler metric and, as $\epsilon\to 0$, $\omega_{t,\epsilon}$ converges to $\omega_{t}+\delta\sqrt{-1}\partial\bar{\partial}|S|_{h}^{2\beta}$ globally on $X$ in the sense of currents and in $C_{loc}^{\infty}(X\setminus D)$-topology. Moreover there exists a uniform constant $C>1$ such that
\begin{equation}\label{fact}
0\leq\chi(\epsilon^{2}+|S|_{h}^{2})\leq C
\end{equation}
and
\begin{equation}
\omega_{0,\epsilon}\geq C^{-1}\omega_{0}.
\end{equation}
\par Consider the following smooth approximation equation of (\ref{CKRF'})
\begin{equation}\label{AMA}
\left\{
\begin{aligned}
\partial_{t}\varphi_{\epsilon}&=\log\frac{e^{t}(|S|_{h}^{2}+\epsilon^{2})^{1-\beta}(\omega_{t,\epsilon}+\sqrt{-1}\partial\bar{\partial}\varphi_{\epsilon})^{2}}{\Omega}-\varphi_{\epsilon}-\delta\chi(|S|_{h}^{2}+\epsilon^{2})\\
\varphi_{\epsilon}(0)&=0,
\end{aligned}
\right.
\end{equation}
which is equivalent to the following generalized K\"{a}hler-Ricci flow
\begin{equation}\label{AKRF}
\left\{
\begin{aligned}
\partial_{t}\omega_{\varphi_{\epsilon}}&=-Ric(\omega_{\varphi_{\epsilon}})-\omega_{\varphi_{\epsilon}}+(1-\beta)\sqrt{-1}\partial\bar{\partial}(|S|_{h}^{2}+\epsilon^{2})+(1-\beta)R_{h}\\ \omega_{\varphi_{\epsilon}}(0)&=\omega_{0,\epsilon},
\end{aligned}
\right.
\end{equation}
where $\omega_{\varphi_{\epsilon}}=\omega_{t,\epsilon}+\sqrt{-1}\partial\bar{\partial}\varphi_{\epsilon}$.
\par Since the twisted canonical line bundle $K_{X}+(1-\beta)L_{D}$ is nef (see assumption (\ref{assumption})), the generalized K\"{a}hler-Ricci flow (\ref{AMA}), or (\ref{AKRF}), has a unique smooth long time solution \cite{C,TZo,Ts}.
\par Set $\psi_{\epsilon}=\varphi_{\epsilon}+\delta\chi(\epsilon^{2}+|S|_{h}^{2})$ and $\omega_{\psi_{\epsilon}}=\omega_{t}+\sqrt{-1}\partial\bar{\partial}\psi_{\epsilon}=\omega_{\varphi_{\epsilon}}$. The following lemma is proved in \cite{Ed}.
\begin{lem}\cite{Ed}\label{Ed'}
There exists a uniform positive constant $C$ such that on $X\times[0,\infty)$,
\begin{equation}
|\psi_{\epsilon}|+|\partial_{t}\psi_{\epsilon}|+tr_{\omega_{\psi_{\epsilon}}}f^{*}\theta'\leq C\nonumber.
\end{equation}
\end{lem}

Following \cite{To10} we fix a smooth nonnegative function $\sigma$ on $X$, which vanishes exactly on $\Gamma$, such that $\sigma\leq1$ and
\begin{equation}\label{fact1}
\sqrt{-1}\partial\sigma\wedge\bar{\partial}\sigma\leq C f^{*}\theta', {}-C f^{*}\theta'\leq\sqrt{-1}\partial\bar{\partial}\sigma\leq C f^{*}\theta'
\end{equation}
on $X$ for some constant $C$.
Define a function on $\Sigma$ by
\begin{equation}
\overline{\psi_{\epsilon}}(s)=\frac{\int_{X_{s}}\psi_{\epsilon}\omega_{0}\mid_{X_{s}}}{\int_{X_{s}}\omega_{0}\mid_{X_{s}}}\nonumber.
\end{equation}
In the following, we will also use $\overline{\psi_{\epsilon}}$ to denote the pull-back $f^*\overline{\psi_{\epsilon}}$ to $X$.

\begin{lem}\label{L1}
There exist positive constants $C$ and $\lambda$ such that
\begin{equation}
|\psi_{\epsilon}-\overline{\psi_{\epsilon}}|\leq\frac{Ce^{-t}}{\sigma^{\lambda}}.
\end{equation}
\end{lem}
\begin{proof}
Note that $\frac{1}{|S|_{h}^{2(1-\beta)}}\leq\frac{1}{\sigma^{\lambda}}$ if we choose $\lambda$ large enough. Then combining the estimates in Lemma \ref{L1} and the uniform boundedness of $\chi(\epsilon^{2}+|S|_{h}^{2})$ mentioned in (\ref{fact}), this lemma can be checked by the same arguments as in the proofs of Corollary 5.1 and Corollary 5.2 of \cite{ST06}.
\end{proof}

\begin{lem}
There exist positive constants $C$ and $\lambda$ such that
\begin{equation}\label{EV1}
(\partial_{t}-\Delta_{\omega_{\psi_{\epsilon}}})(e^{t}(\psi_{\epsilon}-\overline{\psi_{\epsilon}}))\geq tr_{\omega_{\psi_{\epsilon}}}\omega_{0}-\frac{C}{\sigma^{\lambda}}-Ce^{t}.
\end{equation}
\end{lem}
\begin{proof}
Firstly recall the inequality (5.22) in \cite{ST06}:
\begin{equation}
\Delta_{\omega_{\psi_{\epsilon}}}(e^{t}(\psi_{\epsilon}-\overline{\psi_{\epsilon}}))\leq-tr_{\omega_{\psi_{\epsilon}}}\omega_{0}+\frac{1}{\int_{X_{s}}\omega_{0}\mid_{X_{s}}}tr_{\omega_{\psi_{\epsilon}}}\left(\int_{X_{s}}\omega_{0}^{2}\right)+2e^{t}\nonumber.
\end{equation}
Moreover,
\begin{align}
\frac{1}{\int_{X_{s}}\omega_{0}\mid_{X_{s}}}tr_{\omega_{\psi_{\epsilon}}}\left(\int_{X_{s}}\omega_{0}^{2}\right)&\leq\frac{C}{\int_{X_{s}}\omega_{0}\mid_{X_{s}}}tr_{\omega_{\psi_{\epsilon}}}\left(\int_{X_{s}}\Omega\right)\nonumber\\
&=\frac{C}{\int_{X_{s}}\omega_{0}\mid_{X_{s}}}\left(\frac{\Omega}{\omega_{SF}\wedge f^{*}\theta'}\right)tr_{\omega_{\psi_{\epsilon}}}\left(\int_{X_{s}}\omega_{SF}\wedge f^{*}\theta'\right)\nonumber\\
&=C\left(\frac{\Omega}{\omega_{SF}\wedge f^{*}\theta'}\right)tr_{\omega_{\psi_{\epsilon}}}f^{*}\theta'\nonumber\\
&\leq\frac{C}{\sigma^{\lambda}}\nonumber
\end{align}
for some uniform positive constants $C$ and $\lambda$. Thus
\begin{equation}\label{EV1.1}
\Delta_{\omega_{\psi_{\epsilon}}}(e^{t}(\psi_{\epsilon}-\overline{\psi_{\epsilon}}))\leq-tr_{\omega_{\psi_{\epsilon}}}\omega_{0}+\frac{C}{\sigma^{\lambda}}+2e^{t}.
\end{equation}
On the other hand,
\begin{align}\label{EV1.2}
\partial_{t}(e^{t}(\psi_{\epsilon}-\overline{\psi_{\epsilon}}))&=e^{t}(\psi_{\epsilon}-\overline{\psi_{\epsilon}})+e^{t}(\partial_{t}\psi_{\epsilon}-\partial_{t}\overline{\psi_{\epsilon}})\nonumber\\
&\geq-\frac{C}{\sigma^{\lambda}}-Ce^{t},
\end{align}
where we have used Lemma \ref{Ed'} and Lemma \ref{L1}. Combining (\ref{EV1.1}) and (\ref{EV1.2}) we conclude (\ref{EV1}).
\end{proof}

Now we are ready to prove the following key lemma in this section.

\begin{lem}\label{keylemma}
There exist positive constants $C$ and $\lambda$ such that
\begin{equation}\label{C2}
tr_{\omega_{\psi_{\epsilon}}}\omega_{t}\leq Ce^{\frac{C}{\sigma^{\lambda}}}.
\end{equation}
\end{lem}

\begin{proof}
Recall that $\omega_{t}=e^{-t}\omega_{0}+(1-e^{-t})f^{*}\theta$. By Lemma \ref{Ed'}, it suffices to show
\begin{equation}\label{C2.1}
tr_{\omega_{\psi_{\epsilon}}}(e^{-t}\omega_{0})\leq Ce^{\frac{C}{\sigma^{\lambda}}}.
\end{equation}
A direct computation gives
\begin{align}\label{EV3}
&(\partial_{t}-\Delta_{\omega_{\psi_{\epsilon}}})tr_{\omega_{\psi_{\epsilon}}}(e^{-t}\omega_{0})\nonumber\\
&\leq C_{1}e^{-t}(tr_{\omega_{\psi_{\epsilon}}}\omega_{0})^{2}-e^{-t}g_{\psi_{\epsilon}}^{\bar{b}i}g_{\psi_{\epsilon}}^{\bar{j}a}(g_{0})_{i\bar{j}}\Theta_{a\bar{b}}-g_{\psi_{\epsilon}}^{\bar{j}i}g_{\psi_{\epsilon}}^{\bar{q}p}(g_{0})^{\bar{d}c}\nabla_{i}(g_{0})_{p\bar{d}}\overline{\nabla}_{j}(g_{0})_{c\bar{q}},
\end{align}
where $C_{1}$ is a positive constant only depending on the bisectional curvature of $\omega_{0}$, and $\Theta=(1-\beta)\sqrt{-1}\partial\bar{\partial}\log(|S|_{h}^{2}+\epsilon^{2})+(1-\beta)R_{h}$. Note that it was shown in \cite{Ed} that there exists a uniform positive constant $C$ such that
\begin{equation}
\Theta\geq-Cf^{*}\theta'\nonumber,
\end{equation}
which implies that, increasing $C$ if necessary,
\begin{equation}
\Theta\geq-C\omega_{0}\nonumber.
\end{equation}
Substituting the above inequality into (\ref{EV3}), we obtain
\begin{align}\label{EV3.1}
&(\partial_{t}-\Delta_{\omega_{\psi_{\epsilon}}})tr_{\omega_{\psi_{\epsilon}}}(e^{-t}\omega_{0})\nonumber\\
&\leq C_{1}e^{-t}(tr_{\omega_{\psi_{\epsilon}}}\omega_{0})^{2}+Ce^{-t}g_{\psi_{\epsilon}}^{\bar{b}i}g_{\psi_{\epsilon}}^{\bar{j}a}(g_{0})_{i\bar{j}}(g_{0})_{a\bar{b}}-g_{\psi_{\epsilon}}^{\bar{j}i}g_{\psi_{\epsilon}}^{\bar{q}p}(g_{0})^{\bar{d}c}\nabla_{i}(g_{0})_{p\bar{d}}\overline{\nabla}_{j}(g_{0})_{c\bar{q}}\nonumber\\
&\leq Ce^{-t}(tr_{\omega_{\psi_{\epsilon}}}\omega_{0})^{2}-g_{\psi_{\epsilon}}^{\bar{j}i}g_{\psi_{\epsilon}}^{\bar{q}p}(g_{0})^{\bar{d}c}\nabla_{i}(g_{0})_{p\bar{d}}\overline{\nabla}_{j}(g_{0})_{c\bar{q}}\nonumber.
\end{align}
Hence we have
\begin{equation}\label{EV3.2}
(\partial_{t}-\Delta_{\omega_{\psi_{\epsilon}}})\log tr_{\omega_{\psi_{\epsilon}}}(e^{-t}\omega_{0})\leq C tr_{\omega_{\psi_{\epsilon}}}\omega_{0}
\end{equation}
for some uniform positive constant $C$.
\par Set $H=\sigma^{\lambda_{1}}(\log tr_{\omega_{\psi_{\epsilon}}}(e^{-t}\omega_{0})-Ae^{t}(\psi_{\epsilon}-\overline{\psi_{\epsilon}}))$ for some large constants $\lambda_{1}$ and $A$. Using (\ref{fact1}) and the evolution inequalities (\ref{EV1}) and (\ref{EV3.2}) we have
\begin{equation}\label{EV3.3}
(\partial_{t}-\Delta_{\omega_{\psi_{\epsilon}}})H\leq -\frac{A}{2}\sigma^{\lambda_{1}}tr_{\omega_{\psi_{\epsilon}}}\omega_{0}+C\sigma^{\lambda_{1}-2}\log tr_{\omega_{\psi_{\epsilon}}}(e^{-t}\omega_{0})-2Re\left(\frac{\nabla H\overline{\nabla}\sigma^{\lambda_{1}}}{\sigma^{\lambda_{1}}}\right)+2Ae^{t}.
\end{equation}
Note that when $tr_{\omega_{\psi_{\epsilon}}}\omega_{0}>1$ and $\sigma>0$,
\begin{equation}
\log tr_{\omega_{\psi_{\epsilon}}}\omega_{0}\leq2\sqrt{tr_{\omega_{\psi_{\epsilon}}}\omega_{0}}\leq\sigma^{2}tr_{\omega_{\psi_{\epsilon}}}\omega_{0}+\frac{1}{\sigma^{2}}\nonumber.
\end{equation}
Substituting the above inequality into (\ref{EV3.3}), we have
\begin{equation}\label{EV3.4}
(\partial_{t}-\Delta_{\omega_{\psi_{\epsilon}}})H\leq -\frac{A}{3}\sigma^{\lambda_{1}}tr_{\omega_{\psi_{\epsilon}}}\omega_{0}-2Re\left(\frac{\nabla H\overline{\nabla}\sigma^{\lambda_{1}}}{\sigma^{\lambda_{1}}}\right)+3Ae^{t}.
\end{equation}
Now by the maximum principle argument we can conclude (\ref{C2.1}). Lemma \ref{keylemma} is proved.
\end{proof}

By Lemma \ref{keylemma} we have
\begin{align}
tr_{\omega_{t}}\omega_{\psi_{\epsilon}}&\leq (tr_{\omega_{\psi_{\epsilon}}}\omega_{t})\frac{\omega_{\psi_{\epsilon}}^{2}}{\omega_{t}^{2}}\nonumber\\
&\leq C (tr_{\omega_{\psi_{\epsilon}}}\omega_{t})\frac{e^{-t}|S|_{h}^{2(1-\beta)}\Omega}{e^{-t}\omega_{0}\wedge f^{*}\theta'}\nonumber\\
&\leq \frac{C}{\sigma^{\lambda}|S|_{h}^{2(1-\beta)}}e^{\frac{C}{\sigma^{\lambda_{1}}}}\nonumber\\
&\leq Ce^{\frac{C}{\sigma^{\lambda_{1}+1}}}\nonumber.
\end{align}
In conclusion, we have obtained that
\begin{equation}\label{C2''}
C^{-1}e^{-\frac{C}{\sigma^{\lambda_{1}}}}\omega_{t}\leq\omega_{\psi_{\epsilon}}\leq Ce^{\frac{C}{\sigma^{\lambda_{1}+1}}}\omega_{t}.
\end{equation}
Now, since all the smooth fibers of $f$ are elliptic curves, we can apply an idea due to \cite{GTZ,HeTo} (see also \cite{FZ,To15}) to obtain the local higher order estimates for $\psi_{\epsilon}$.
\begin{prop}\label{Cinfty}
For any fixed $K\subset\subset X\setminus\Gamma$ and all $k\in\mathbb{N}$, there exists a positive constant $C_{K,k}$, which is independent of $\epsilon$ and $t$, such that
\begin{equation}\label{Cinfty'}
\|\psi_{\epsilon}\|_{C^{k}(K,\omega_{0})}\leq C_{K,k}.
\end{equation}
\end{prop}
\begin{proof}
Note that on the approximation equation (\ref{AKRF}), the extra term $(1-\beta)\sqrt{-1}\partial\bar{\partial}(|S|_{h}^{2}+\epsilon^{2})+(1-\beta)R_{h}=f^*((1-\beta)\sqrt{-1}\partial\bar{\partial}(|S'|_{h'}^{2}+\epsilon^{2})+(1-\beta)R_{h'})$, which in particular implies that it does not depend on the variable in fiber direction. Thus we can apply the arguments in Theorem 5.24 of \cite{To15} to conclude this proposition.
\end{proof}

Now we choose a sequence $\epsilon_{j}\to 0^{+}$ such that $\varphi_{\epsilon_{j}}$ converges to the unique solution $\varphi$ of the conical K\"{a}hler-Ricci flow (\ref{CKRF'}) in $L^{1}(X,\omega_{0})$ and in $C_{loc}^{\infty}(X\setminus D)$-topology. From Proposition \ref{Cinfty} we have
\begin{cor}\label{Cinfty.1}
For any fixed $K\subset\subset X\setminus\Gamma$ and all $k\in\mathbb{N}$, there exists a positive constant $C_{K,k}$, such that for all $t\geq0$,
\begin{equation}\label{Cinfty''}
\|\varphi+\delta|S|_{h}^{2\beta}\|_{C^{k}(K,\omega_{0})}\leq C_{K,k}.
\end{equation}
\end{cor}
Now we are going to show the uniform convergence at the level of potentials.
\begin{prop}\label{C0}
For any fixed $K\subset\subset X\setminus\Gamma$, $\varphi\to f^{*}\varphi_{\infty}$ in $L^{\infty}(K)$ as $t\to\infty$, where $\varphi_\infty$ is the function given by Proposition \ref{prop1}.
\end{prop}
\begin{proof}
The proof will make use of a argument due to \cite{ST06}. Firstly, by pulling back a defining section of $s_{1}+\ldots+s_{k}+r$ on $\Sigma$, we may fix a defining section $\tilde{S}$ of the divisor $f^*([s_{1}]+\ldots+[s_{k}]+[r])$ vanishing exactly on $\Gamma$. Moreover, if we denote the holomorphic line bundle on $\Sigma$ associated to $s_{1}+\ldots+s_{k}+r$ by $L'$, then by using $H^{1,1}(\Sigma,\mathbb{R})=\mathbb{R}$ we know $c_1(L')=\mu[\theta']$ for some constant $\mu$, where $[\theta']$ is the K\"ahler class of $\theta'$. Thus by pulling back a smooth Hermitian metric on $L'$, we may fix a smooth Hermitian metric $\tilde{h}$ on the holomorphic line bundle associated to $f^*([s_{1}]+\ldots+[s_{k}]+[r])$ with $Ric(\tilde{h})=\mu f^{*}\theta'$ for some $\mu\in\mathbb{R}$ and $|\tilde{S}|_{\tilde{h}}^{2}\leq1$.
\par Define
\begin{equation}
B_{d}(\Gamma')=\{s\in\Sigma|dist_{\theta}(s,\Gamma')\leq d\}\nonumber
\end{equation}
and
\begin{equation}
B_{d}(\Gamma)=f^{-1}(B_{d}(\Gamma'))\nonumber.
\end{equation}
Let $\varepsilon$ be an arbitrary positive constant. Choose $d_{K}>0$ such that $\overline{K}\subset X\setminus B_{d_{K}}(\Gamma)$ and for all $z\in B_{d_{K}}(\Gamma)$ and $t\geq0$,
\begin{equation}
(\varphi-f^{*}\varphi_{\infty}+\varepsilon\log|\tilde{S}|_{\tilde{h}}^{2})(z,t)<-1\nonumber
\end{equation}
and
\begin{equation}
(\varphi-f^{*}\varphi_{\infty}-\varepsilon\log|\tilde{S}|_{\tilde{h}}^{2})(z,t)>1\nonumber.
\end{equation}
Let $\eta_{K}$ be a smooth cut-off function on $\Sigma$ such that $\eta_{K}=1$ on $\Sigma\setminus B_{d_{K}}(\Gamma')$ and $\eta_{K}=0$ on $B_{\frac{d_{K}}{2}}(\Gamma')$.
\par Define $\rho_{K}=(f^{*}\eta_{K})\rho_{SF}$ and $\omega_{SF,K}=\omega_{0}+\sqrt{-1}\partial\bar{\partial}\rho_{K}$. Then set
\begin{equation}
\psi_{K}^{-}=\varphi-f^{*}\varphi_{\infty}-e^{-t}\rho_{K}+\varepsilon\log|\tilde{S}|_{\tilde{h}}^{2}\nonumber
\end{equation}
and
\begin{equation}
\psi_{K}^{+}=\varphi-f^{*}\varphi_{\infty}-e^{-t}\rho_{K}-\varepsilon\log|\tilde{S}|_{\tilde{h}}^{2}\nonumber.
\end{equation}
A direct computation gives
\begin{equation}\label{EV4}
\partial_{t}\psi_{K}^{-}=\log\frac{e^{t}(e^{-t}\omega_{SF,K}+f^{*}\omega_{\infty}+(\varepsilon\mu-e^{-t})f^{*}\theta'+\sqrt{-1}\partial\bar{\partial}\psi_{K}^{-})^{2}}{2\omega_{SF}\wedge f^{*}\omega_{\infty}}-\psi_{K}^{-}+\varepsilon\log|\tilde{S}|_{\tilde{h}}^{2}.
\end{equation}
Now we can use the same argument as in the proof of Proposition 6.1 of \cite{ST06} to conclude that there exist uniform constants $C$ and $T_{1}$ such that for all $t\geq T_{1}$,
\begin{equation}
\psi_{K}^{-}\leq C\varepsilon\nonumber
\end{equation}
if $\varepsilon$ is sufficiently small. In particular, on $K\times[T_{1},\infty)$,
\begin{equation}\label{C0.1}
\varphi-f^{*}\varphi_{\infty}\leq C_{K}\varepsilon
\end{equation}
for some constant $C_K$ only depending on $K$. Similarly,
\begin{equation}\label{C0.2}
\varphi-f^{*}\varphi_{\infty}\geq-C_{K}\varepsilon
\end{equation}
on $K\times[T_{1},\infty)$. Proposition \ref{C0} follows from (\ref{C0.1}) and (\ref{C0.2}).
\end{proof}

Finally, we are able to prove Theorem \ref{main}.\\
\begin{proof}[Proof of Theorem \ref{main}]
Combining Proposition \ref{prop1}, Corollary \ref{Cinfty.1} and Proposition \ref{C0}, Theorem \ref{main} is proved.
\end{proof}

\section*{Acknowledgements}
The author is grateful to Professor Huai-Dong Cao for constant encouragement and support, as well as many enlighten discussions and suggestions. His thanks also go to Professors Valentino Tosatti and Zhenlei Zhang for very helpful discussions, in particular to Professor Tosatti for reading carefully a previous version of this paper, pointing out a mistake and useful suggestions. The author also thanks referees for some helpful comments.

\end{document}